\documentclass{article}
\usepackage{fullpage}
\usepackage{enumerate}
\usepackage{amsmath,amssymb,amsthm}
\usepackage{graphics}
\usepackage{graphicx}
\usepackage{subfigure}
\usepackage{listings}
\usepackage[linesnumbered, ruled]{algorithm2e}

\def\DB{{\rm DB}}
\def\eps{\varepsilon}

 \newtheorem{theorem}{Theorem}

 \newtheorem{lemma}[theorem]{Lemma}

\newcommand{\answerCommand}{}%
  {\renewcommand{\answerCommand}{#1\\}%
   \noindent\textbf{\answerCommand}%
   }{\\}%

  {\renewcommand{\answerCommand}{#1}%
   \noindent\textbf{\answerCommand}%
   }{\\}%

\title{A computational approach to Conway's thrackle conjecture}
\author{Radoslav Fulek\thanks{Ecole Polytechnique F\'ed\'erale de Lausanne. Email:~\texttt{radoslav.fulek@epfl.ch}.}
 \and J\'anos Pach\thanks{Ecole Polytechnique F\'ed\'erale de Lausanne and City College, New York. Email:~\texttt{pach@cims.nyu.edu}.
Research partially supported by NSF grant CCF-08-30272, grants
from OTKA, SNF, and PSC-CUNY.}}
\date{}

\begin{document}

\maketitle

\thispagestyle{empty}

 \begin{abstract}
 A drawing of a graph in the plane is called a {\em thrackle} if
 every pair of edges meets precisely once, either at a common
 vertex or at a proper crossing. Let $t(n)$ denote the maximum
 number of edges that a thrackle of $n$ vertices can have.
 According to a 40 years old conjecture of Conway, $t(n)=n$ for every
 $n\ge 3$. For any $\eps>0$, we give an algorithm terminating in
 $e^{O((1/\eps^2)\ln(1/\eps))}$ steps to decide whether $t(n)\le (1+\eps)n$ for all $n\ge 3$. Using this approach, we improve the best known upper
 bound, $t(n)\le \frac 32(n-1)$, due to Cairns and Nikolayevsky, to
 $\frac{167}{117}n<1.428n$.
 \end{abstract}











\section{Introduction}

A {\em drawing} of a graph (or a {\em topological graph}) is a
representation of the graph in the plane such that the vertices
are represented by distinct points and the edges by (possibly
crossing) simple continuous curves connecting the corresponding
point pairs and not passing through any other point representing a
vertex. If it leads to no confusion, we make no notational
distinction between a drawing and the underlying abstract graph
$G$. In the same vein, $V(G)$ and $E(G)$ will stand for the vertex
set and edge set of $G$ as well as for the sets of points and
curves representing them.

A drawing of $G$ is called a {\em thrackle} if every pair of edges
meet precisely once, either at a common vertex or at a proper
crossing. (A crossing $p$ of two curves is {\em proper} if at $p$
one curve passes from one side of the other curve to its other
side.) More than {\em forty} years ago Conway~\cite{Conway, BMP,
Ring} conjectured that every thrackle has at most as many edges
as vertices, and offered a bottle of beer for a solution. Since
then the prize went up to a thousand dollars. In spite of
considerable efforts, Conway's thrackle conjecture is still open.
It is believed to represent the tip of an ``iceberg,'' obstructing
our understanding of crossing patterns of edges in topological
graphs. If true, Conway's conjecture would be tight as any cycle
of length at least {\em five} can be drawn as a thrackle, see
\cite{Woodall}. Two thrackle drawings of $C_5$ and $C_6$ are shown
in Figure \ref{fig:c5c6}.

\begin{figure}[h]
\centering
\includegraphics[scale=0.3]{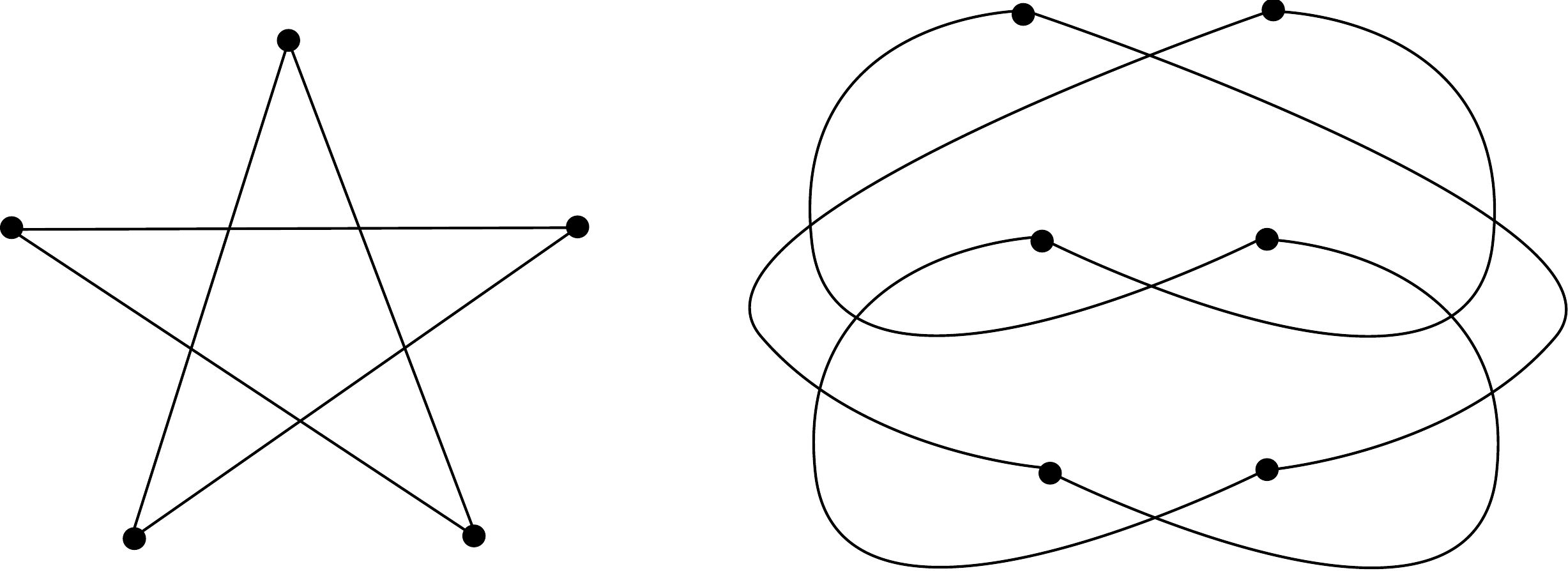}
\caption{$C_5$ and $C_6$ drawn as thrackles} \label{fig:c5c6}
\end{figure}

Obviously, the property that $G$ can be drawn as a thrackle is
{\em hereditary}: if $G$ has this property, then any subgraph of
$G$ does. It is very easy to verify (cf. \cite{Woodall}) that
$C_4$, a cycle of length {\em four}, cannot be drawn in a
thrackle. Therefore, every ``thrackleable" graph is $C_4$-free,
and it follows from extremal graph theory that every thrackle of
$n$ vertices has at most $O(n^{3/2})$ edges \cite{Diestel}. The
first linear upper bound on the maximum number of edges of a
thrackle of $n$ vertices was given by Lov\'asz et
al.~\cite{Lovasz}. This was improved to a $\frac 32(n-1)$ by
Cairns and Nikolayevsky~\cite{Cairns}.

The aim of this note is to provide a finite approximation scheme
for estimating the maximum number of edges that a thrackle of $n$
vertices can have. We apply our technique to improve the best
known upper bound for this maximum.

\smallskip

To state our results, we need a definition. Given three integers
$c',c''>2$, $l\ge 0$, the {\em dumbbell} $\DB(c',c'',l)$ is a
simple graph consisting of two disjoint cycles of length $c'$ and
$c''$, connected by a path of length $l$. For $l=0$, the two
cycles share a vertex. It is natural to extend this definition to
negative values of $l$, as follows. For any $l>-\min(c',c'')$, let
$\DB(c',c'',l)$ denote the graph consisting of two cycles of
lengths $c'$ and $c''$ that share a path of length $-l$. That is,
for any $l>-\min(c',c'')$, we have
$$|V(\DB(c',c'',l))|=c'+c''+l-1.$$
The three types of dumbbells (for $l<0$, $l=0$, and $l>0$) are
illustrated in Figure~\ref{fig:db66}.

\bigskip

\begin{figure}[h]
\centering
\includegraphics[scale=0.5]{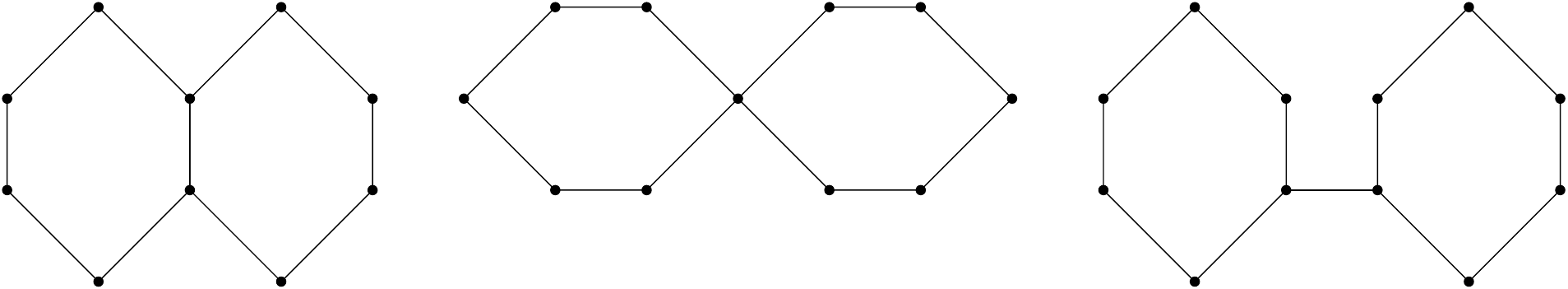}
\caption{Dumbbells $\DB(6,6,-1), \DB(6,6,0),$ and $\DB(6,6,1)$}
\label{fig:db66}
\end{figure}

\medskip

Our first theorem shows that for any $\eps>0$, it is possible to
prove Conway's conjecture up to a multiplicative factor of
$1+\eps$, by verifying that no dumbbell smaller than a certain
size depending on $\eps$ is thrackleable.

\begin{theorem}
\label{thm:Approaching bound} Let $c\geq 6$  and $ l\geq -1$ be two
integers, such that $c$ is even, with the property that no dumbbell $\DB(c',c'',l')$ with
$-c'/2 \leq l'\leq l$ and with even $6 \leq c',c''\leq c$ can be
drawn in the plane as a thrackle. Let $r=\lfloor l/2\rfloor$. Then
the maximum number of edges $t(n)$ that a thrackle on $n$ vertices
can have satisfies $t(n)\le \tau(c,l)n$, where
\begin{displaymath}
  \tau(c,l) = \left\{ \begin{array}{ll}
        \frac{47c^2+116c+80}{35c^2+68c+32}&  {\mbox {if }}  l=-1, \\ \\
        1+\frac{2c^2r+4cr^2+22cr+7c^2+22c+8r^2+24r+16}{2c^2r^2+
            14c^2r+4cr^2 +16cr+24c^2+12c} &  {\mbox {if }} l\ge 0, \\
                \end{array}
                 \right.
\end{displaymath}
as $n$ tends to infinity.
\end{theorem}

As both $c$ and $l$ get larger, the constant $\tau(c,l)$ given by
the second part of Theorem~\ref{thm:Approaching bound} approaches
$1$. On the other hand, assuming that Conway's conjecture is true
for all bipartite graphs with up to 10 vertices, which will be verified in Section~\ref{proof2nd}, the first part of the theorem applied with
$c=6, l=-1$ yields that $t(n)\le \frac{617}{425}n<1.452n$. This bound is already better than the bound $\frac 32 n$ established in \cite{Cairns}.

By a more careful application of Theorem~\ref{thm:Approaching
bound}, i.e. taking $c=6$ and $l=0$, we obtain an even stronger result.

\begin{theorem}
\label{thm:betterBound} The maximum number of edges $t(n)$ that a
thrackle on $n$ vertices can have satisfies the inequality
$t(n)\le \frac{167}{117}n<1.428n.$
\end{theorem}

Our method is algorithmic. We design an $e^{O((1/\eps^2)\ln(1/\eps))}$ time algorithm to prove, for any $\eps>0$, that $t(n)\le (1+\eps)n$ for all $n$, or to exhibit a counterexample to Conway's conjecture. The proof of Theorem~\ref{thm:betterBound} is computer assisted: it requires testing the planarity of certain relatively small graphs.

For thrackles drawn by straight-line edges, Conway's conjecture
had been settled in a slightly different form by Hopf and
Pannwitz \cite{HoP34} and by Sutherland \cite{Su}  before Conway was even born,
 and later, in the above form, by Erd\H os and
Perles. Assuming that Conway's conjecture is true, Woodall
\cite{Woodall} gave a complete characterization of all graphs that
can be drawn as a thrackle. He also observed that it would be
sufficient to verify the conjecture for dumbbells. This
observation is one of the basic ideas behind our arguments.

Several interesting special cases and variants of the conjecture
are discussed in \cite{Cairns, Cairns3, Cairns2, Green, Lovasz,
PePi, PiRi}.

In Section~\ref{doubling}, we describe a crucial construction of Conway and summarize some earlier results needed for our arguments. The proofs of Theorems~\ref{thm:Approaching bound} and~\ref{thm:betterBound} are given in Sections~\ref{proofmain} and~\ref{proof2nd}. The analysis of the algorithm for establishing the $(1+\eps)n$ upper bound for the maximum number of edges that a thrackle of $n$ vertices can have is also given in Section~\ref{proof2nd} (Theorem~\ref{thm:runningTime}). In the last section, we discuss some related Tur\'an-type extremal problems for planar graphs.

\section{Conway's doubling and preliminaries}\label{doubling}

In this section, we review some earlier results that play a key
role in our arguments.

A {\em generalized thrackle} is a drawing of a graph in the plane
with the property that any pair of edges share an odd number of
points at which they properly cross or which are their common
endpoints. Obviously, every thrackle is a generalized thrackle but
not vice versa: although $C_4$ is not thrackleable, it can be
drawn as a generalized thrackle, which is not so hard to see.




We need the following simple observation based on the Jordan curve
theorem.

\begin{lemma}
\label{lemma:removing odd cycle} {\rm \cite{Lovasz}} A
(generalized) thrackle cannot contain two vertex disjoint odd
cycles.
\end{lemma}

Lov\'asz, Pach, and Szegedy~\cite{Lovasz} gave a somewhat
counterintuitive characterization of generalized thrackles
containing no odd cycle: a bipartite graph is a generalized
thrackle if and only if it is {\em planar}. Moreover, it follows
immediately from Lemma 3 and the proof of Theorem 3 in Cairns and
Nikolayevsky~\cite{Cairns} that this statement can be strengthened
as follows.

\begin{lemma}
 \label{lemma:from thrackle to planarity} {\rm \cite{Cairns}} Let $G$ be a bipartite graph with vertex set $V(G)=A\cup B$ and edge set $E(G)\subseteq A\times B$. If $G$ is a generalized thrackle then it can be redrawn in the plane without crossing so that the cyclic order of the edges around any vertex $v\in V(G)$ is preserved if $v\in A$ and reversed if $v\in B$.
 \end{lemma}

We recall a construction of Conway for transforming a thrackle
into another one. It can be used to eliminate odd cycles.

Let $G$ be a thrackle or a generalized thrackle that contains an
{\em odd} cycle $C$. In the literature, the following procedure is
referred to as {\em Conway's doubling}: First, delete from $G$ all
edges incident to at least one vertex belonging to $C$, including
all edges of $C$. Replace every vertex $v$ of $C$ by two nearby
vertices, $v_1$ and $v_2$. For any edge $vv'$ of $C$, connect
$v_1$ to $v'_2$ and $v_2$ to $v'_1$ by two edges running very
close to the original edge $vv'$, as depicted in Figure
\ref{fig:Conway doubling}. For any vertex $v$ belonging to $C$,
the set of edges incident to $v$ but not belonging to $C$ can be
divided into two classes, $E_1(v)$ and $E_2(v)$: the sets of all
edges whose initial arcs around $v$ lie on one side or the other
side of $C$. In the resulting topological graph $G'$, connect all
edges in $E_1(v)$ to $v_1$ and all edges in $E_2(v)$ to $v_2$ so
that every edge connected to $v_1$ crosses all edges connected to
$v_2$ exactly once in their small neighborhood. See Figure
\ref{fig:Conway doubling}. All other edges of $G$ remain
unchanged. Denote the vertices of the original odd cycle $C$ by
$v^1, v^2,\ldots, v^k$, in this order. In the resulting drawing
$G'$, we obtain an {\em even} cycle $C'=v^1_1
v^2_2v^3_1v^4_2\ldots v^1_2v^2_1v^3_2v^4_1\ldots$ instead of $C$.
It is easy to verify that $G'$ is drawn as a thrackle, which is
stated as part (ii) of the following lemma (see also Lemma 2
in~\cite{Cairns}).

\bigskip

\begin{figure}[h]
\centering
\includegraphics[scale=0.5]{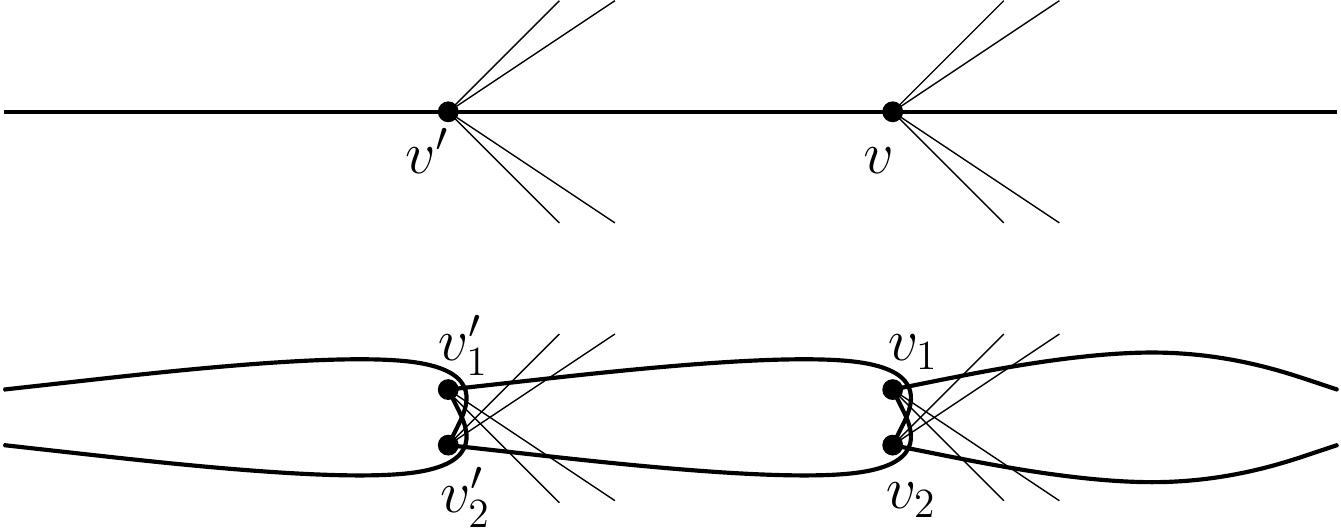}
\caption{Conway's doubling of a cycle} \label{fig:Conway doubling}
\end{figure}

\medskip

\begin{lemma}
 \label{lemma:conway doubling} {\rm (Conway, \cite{Woodall, Cairns})}
 Let $G$ be a (generalized) thrackle with at least one odd cycle $C$. Then the topological graph $G'$ obtained from $G$ by Conway's doubling of $C$ is

 (i) bipartite, and

 (ii) a (generalized) thrackle.
 \end{lemma}

\begin{proof}
It remains to verify part (i). Let $k$ denote the length of the
(odd) cycle $C\subseteq G$, and let $C'$ stand for the doubled
cycle in $G'$. The length of $C'$ is $2k$. Let $\pi$ denote the
inverse of the doubling transformation. That is, $\pi$ identifies
the opposite pairs of vertices in $C'$, and takes $C'$ into $C$.

Suppose for a contradiction that $G'$ is not bipartite. In view of
Lemma \ref{lemma:removing odd cycle}, no odd cycle of $G'$ is
disjoint from $C'$. Let $D'$ be an odd cycle in $G'$ with the
smallest number of edges that do not belong to $C'$. We can assume
that $D'$ is the union of two paths, $P_1$ and $P_2$, connecting
the same pair of vertices $u,v$ in $C'$, where $P_1$ belongs to
$C'$ and $P_2$ has no interior points on $C'$.

If $\pi(u)\not=\pi(v)$, that is, the length of $P_1$ is not $0$ or
$k$, then $\pi(D')=\pi(P_1)\cup\pi(P_2)$ is a simple cycle in $G$.
Notice that the lengths of $P_1$ and $P_2$ have different
parities. If the length of $P_1$ is even, say, then, according to
the rules of doubling, the initial and final pieces of $P_2$ in
small neighborhoods of $u$ and $v$ are on the same side of the
(arbitrarily oriented) cycle $C'$. Consequently, the initial and
final pieces of $\pi(P_2)$ in small neighborhoods of $\pi(u)$ and
$\pi(v)$ are on the {\em same} side of $C$. On the other hand,
using the fact that $G$ is a generalized thrackle, the total
number of intersection points between the odd path $\pi(P_2)$ and
the odd cycle $C$ is odd (see the proof of Lemma 2.2 from
\cite{Lovasz}). Thus, if we two color the regions of the plane
bounded by pieces of $C$, so that any pair of neighboring regions receive different colors,  the initial and final pieces of
$\pi(P_2)$ in small neighborhoods of $\pi(u)$ and $\pi(v)$ must
lie in the regions colored with different colors. Since $C$ is odd and drawn as a generalized thrackle, it follows that 
the initial and final pieces of
$\pi(P_2)$ in small neighborhoods of $\pi(u)$ and $\pi(v)$ must lie on {\em different} sides of $C$, a contradiction.

The cases when $P$ is odd and when $\pi(u)=\pi(v)$ can be treated
analogously.
\end{proof}

Finally, we recall an observation of Woodall \cite{Woodall}
mentioned in the introduction, which motivated our investigations.

As thrackleability is a hereditary property, a minimal
counterexample to the thrackle conjecture must be a connected
graph $G$ with exactly $|V(G)|+1$ edges and with no vertex of
degree {\em one}. Such a graph $G$ is necessarily a dumbbell
$\DB(c',c'',l)$. If $l\not=0$, then $G$ consists of two cycles
that share a path or are connected by a path $uv$. In both cases,
we can ``double'' the path $uv$, as indicated in Figure
\ref{fig:singledoubling}, to obtain another thrackle $G'$. It is
easy to see that $G'$ is a dumbbell consisting of two cycles that
share precisely one vertex (the vertex $v$ in the figure). Moreover, if any of these two cycles
is not even, then we can double it and repeat the above procedure,
if necessary, to obtain a dumbbell $\DB(b',b'',0)$ drawn as a
thrackle, where $b'$ and $b''$ are even numbers.

Thus, in order to prove the thrackle conjecture, it is enough to
show that no dumbbell $\DB(c',c'',0)$ consisting of two even
cycles that share a vertex is thrackleable.

\bigskip

\begin{figure}[h]
\centering
\includegraphics[scale=0.5]{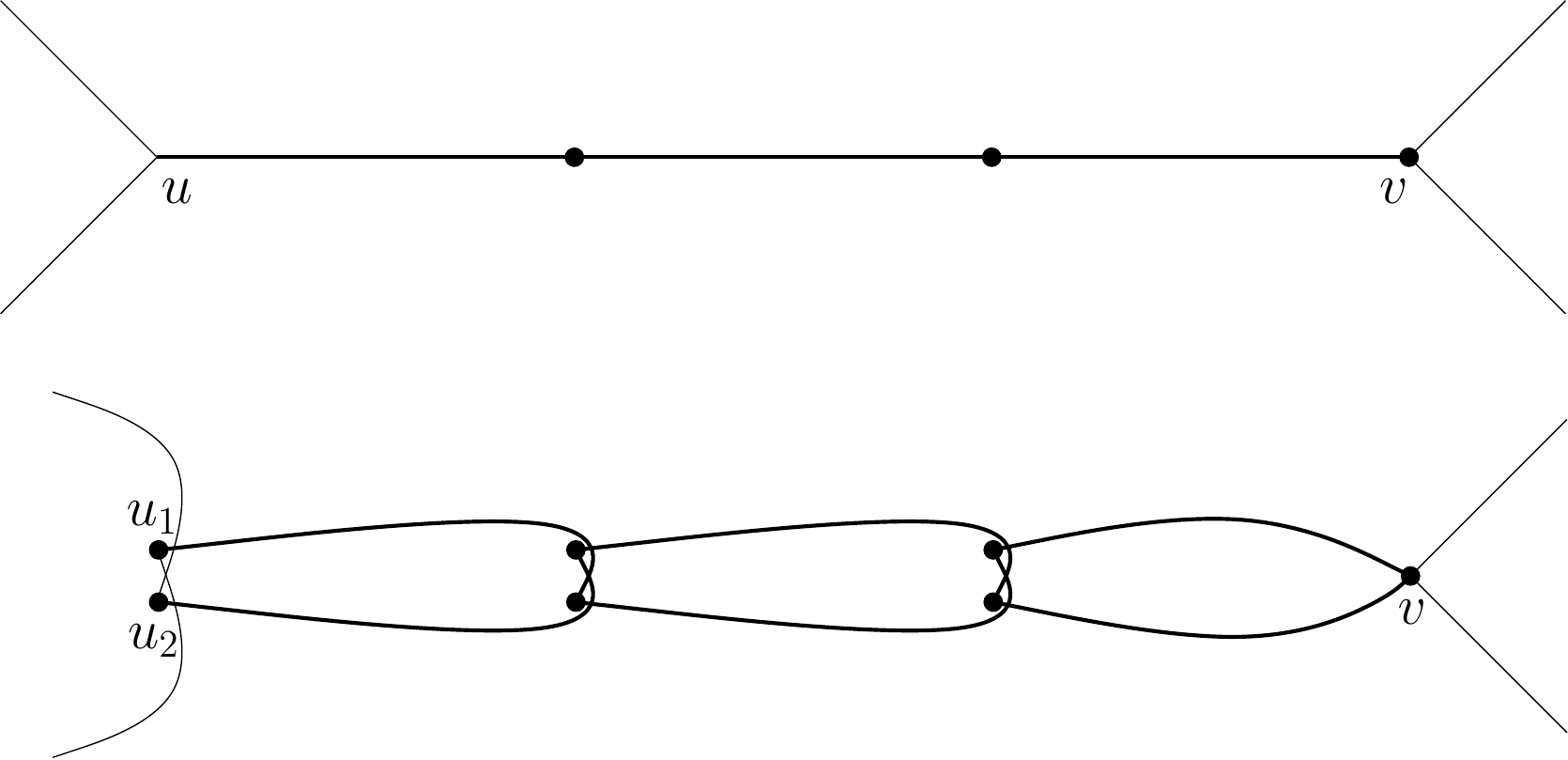}
\caption{Doubling the path $uv$} \label{fig:singledoubling}
\end{figure}

\medskip
 \section{Proof of Theorem \ref{thm:Approaching bound} }\label{proofmain}

Let $c\geq 6$ and $ l\geq -1$ be two integers, and suppose that no
dumbbell $\DB(c',c'',l')$ with $-c'/2 \leq l'\leq l$ and with even
$6 \leq c',c''\leq c$ can be drawn in the plane as a thrackle. For
simpler notation, let $r=\lfloor l/2\rfloor$.

Let $G=(V,E)$ be a thrackleable graph with $n$ vertices and $m$
edges. We assume without loss of generality that $G$ is connected
and that it has no vertex of degree {\em one}. Otherwise, we can
successively delete all vertices of degree {\em one}, and argue
for each connected component of the resulting graph separately.

As usual, we call a graph {\em two-connected} if it is connected and
it has no {\em cut vertex}, i.e., it cannot be separated into two
or more parts by the removal of a vertex \cite{Diestel}.

We distinguish three cases:

\begin{enumerate}[(A)]
\item
$G$ is bipartite;
\item
$G$ is not bipartite, and the graph $G'$ obtained by performing
Conway's doubling of a shortest odd cycle $C\subset G$ is
2-connected;
\item
$G$ is not bipartite, and the graph $G'$ obtained by performing
Conway's doubling of a shortest odd cycle $C\subset G$ is
not 2-connected.
\end{enumerate}
In each case, we will prove that $m\le\tau(c,l)n$.  
\medskip

(A) {By Lemma \ref{lemma:from thrackle to planarity}, in this case $G$
is planar. We fix an embedding of $G$ in the plane. According to
the assumption of our theorem, $G$ contains no subgraph that is a
dumbbell $\DB(c',c'',l')$, for any even $6 \leq c'\leq c''\leq c$,
and $-c'/2 \leq l'\leq l$. We also know that $G$ has no $C_4$. We
are going to use these conditions to bound the number of edges
$m=|E(G)|$.

Notice that we also exclude dumbbells
$\DB(c',c'',l')$ with $-c'\leq l'<-c'/2$. Indeed, in this case $\DB(c',c'',l')$ is isomorphic to
$\DB(c',d, k)$, where $d=(c'+c''+2l')$, $k=(-c'-l')$, and $d<c''\le
c$, $\max(-c'/2,-d/2)\le k<0$.  		


Suppose first that $G$ is {\em two-connected}. Let $f$ denote the number of faces, and let $f_c$ stand for the number of faces with at most $c$ sides. By double counting the edges, we obtain

\begin{equation}
\label{eqn:Approaching bound100} 2m\geq 6f_c+ (c+2)(f-f_c).
\end{equation}

If $l=-1$, then applying the condition on forbidden dumbbells, we
obtain that no two faces of size at most $c$ share an edge, so
that $6f_c\le m$. If $l\ge 0$, Menger's theorem implies that any
two faces of size at most $c$ are connected by two vertex disjoint
paths. Since any such path must be longer than $l$, to each face
we can assign its vertices as well as the $r=\lfloor l/2\rfloor$
closest vertices along two vertex disjoint paths leaving the face,
and these sets are disjoint for distinct faces. Thus, we have
$f_c(2r+6)\le n$. In either case, we have

\begin{equation}
\label{eqn0}
f_c \leq \left\{ \begin{array}{ll}
                \frac{m}{6} &   {\mbox {if }}   l=-1, \\ \\
                \frac{n}{2r+6} &  {\mbox {if }}    l\geq 0. \\
                \end{array}
                 \right.
\end{equation}

Combining the last two inequalities, we
obtain

\begin{displaymath}\label{eqn1}
f \le \frac{(c-4)f_c+2m}{c+2}\leq
\left\{ \begin{array}{ll}
                 \frac{(c-4)\frac{m}{6}+2m}{c+2} &  {\mbox {if }}   l=-1,
                 \\  \\
                 \frac{(c-4)\frac{n}{2r+6}+2m}{c+2} &  {\mbox {if }}   l\geq 0. \\
                \end{array}
                 \right.
\end{displaymath}

In view of Euler's polyhedral formula $m+2=n+f$, which yields

\begin{equation}
\label{eqn:Approaching bound1}
  m \leq \left\{ \begin{array}{ll}
                \frac{6c+12}{5c+4}n-\frac{12c+24}{5c+4}&  {\mbox {if }}   l=-1,
                \\ \\
                \frac{2cr+4r+7c+8}{2cr+6c}n - \frac{2c+4}{c} &  {\mbox {if }}   l\geq 0. \\
                \end{array}
                 \right.
\end{equation}

It can be shown by routine calculations that the last estimates,
even if we ignore their negative terms independent of $n$, are
stronger than the ones claimed in the theorem. (In fact, they are
also stronger than the corresponding bounds (\ref{eqn:Approaching
bound7}) and (\ref{eqn:Approaching bound5}) in Case (B); see
below.) This concludes the proof of the case (A) when $G$ is 2-connected.

If $G$ is not 2-connected, then consider a block decomposition of
$G$, and proceed by induction on the number of blocks. The base case, i.e
when $G$ is 2-connected, is treated above. Otherwise $G$ can be obtained as a union of
two bipartite graphs $G_1=(V_1, E_1)$ and $G_2=(V_2,E_2)$ sharing exactly one vertex. By induction hypothesis 
we can use (\ref{eqn:Approaching bound1}) to bound the number of edges in $G_i$, for $i=1,2$,
by substituting $|E_i|$ and $|V_i|$ for $m$ and $n$, respectively.
We obtain the claimed bound on the maximum number of edges in $G$ by adding up the bounds on $|E_1|$ and $|E_2|$ as follows.

$$|E(G)| = |E(G_1)|+|E(G_2)|\leq k_1|V(G_1)|+k_1|V(G_2)|-2k_2=k_1|V(G)|+k_1-2k_2$$ where $k_1=k_1(c,l)$ and $k_2=k_2(c,l)$ represent 
the constants in (\ref{eqn:Approaching bound1}).
 Induction goes through, because $k_1<k_2$ for all considered values of $c$ and $l$.\\
\label{case:A}
}

{\label{case:B}
(B) In this case, we establish two upper bounds on the maximum number
of edges in $G$: one that decreases with the length of the shortest
odd cycle $C\subseteq G$ and one that increases. Finally, we
balance between these two bounds.

By doubling a shortest odd cycle $C\subseteq G$, as before, we obtain a bipartite thrackle $G'$ (see Lemma~\ref{lemma:conway doubling}). Let $C'$ denote the doubled cycle in $G'$. By Lemma \ref{lemma:from thrackle to planarity}, $G'$ is a two-colorable planar graph. Moreover, it can be embedded in the plane without crossing so that the cyclic order of the edges around each vertex in one color class is preserved, and for each vertex in the other color class reversed. A closer inspection of the way how we double $C$ shows that as we traverse $C'$ in $G'$, the edges incident to $C'$ start on alternating sides of $C'$. This implies that, after redrawing $G'$ as a plane graph, all edges incident to $C'$ lie on one side, that is, $C'$ is a {\em face}.

Slightly abusing the notation, from now on let $G'$ denote a crossing-free drawing with the above property, which has a $2|C|$-sided face $C'$. Denoting the number of vertices and edges of $G'$ by $n'$ and $m'$, the number of faces and the number of faces of size at most $c$ by $f'$ and $f'_c$, respectively, we have $n'=n+|C|=|V(G')|$,
$m'=m+|C|=|E(G')|$, and, as in Case (A), inequality~(\ref{eqn0}),

\begin{displaymath}
\label{eqn:Approaching bound2}
  f_c' \leq \left\{ \begin{array}{ll}
                \frac{1}{6}m' &   {\mbox {if }}   l=-1,\\ \\
                \frac{n'}{2r+6} &  {\mbox {if }}    l\geq 0. \\
                \end{array}
                 \right.
\end{displaymath}

Double counting the edges of $G'$, we obtain
\begin{displaymath}
2m'\geq 6f_c'+ (c+2)(f'-1-f_c') + 2|C|.
\end{displaymath}

In case $l\geq 0$, combining the last two inequalities, we have
\begin{displaymath}
 f' \leq \frac{(c-4)f_c'+ 2(m'-|C|)+c+2}{c+2}  \leq \frac{(c-4)\frac{n'}{2r+6}+2(m'-|C|)+c+2}{c+2}.
\end{displaymath}

By Euler's polyhedral formula, $f'=m'-n'+2$. Thus, after ignoring the negative term, which depends only on $c$ and $l$, the last inequality yields
\begin{equation}
\label{eqn:Approaching bound5} |E(G)|\leq \frac{2cr + 4r+ 7c+8}{2cr+
6c}n+|C|\frac{c-4}{2cr+ 6c}.
\end{equation}

The case $l=-1$ can be treated analogously, and the corresponding
bound on $E(G)$ becomes
\begin{equation}
\label{eqn:Approaching bound7} |E(G)|\leq \frac{ 6c +
12}{5c+4}n+|C|\frac{c-4}{5c+4}.
\end{equation}

We now establish another upper bound on the number of edges in $G$: one that decreases with the length of the shortest odd cycle $C$ in $G$. As in \cite{Lovasz}, we remove from $G$ the vertices of $C$ together with all edges incident to them. Let $G''$ denote the resulting thrackle. By Lemma \ref{lemma:removing odd cycle}, $G''$ is bipartite. By Lemma \ref{lemma:from thrackle to planarity}, it is a planar graph. From now on, let $G''$ denote a fixed (crossing-free) embedding of this graph. According to our assumptions, $G''$ has no subgraph isomorphic to $\DB(c',c'',l')$, for any even numbers $c'$ and $c''$ with $6 \leq c'\leq c''\leq c$, and for any integer $l'$ with $-c'/2 \leq l'\leq l$.

We can bound $|E(G'')|$, as follows. By the minimality of $C$, each vertex $v\in V(G)$ that does not belong to $C$ is joined by an edge of $G$ to at most one vertex on $C$. Indeed, otherwise, $v$ would create either a $C_4$ or an odd cycle shorter than $C$. Hence, if $l\geq 0$, inequality (\ref{eqn:Approaching bound1}) implies that
\begin{equation}
\label{eqn:Approaching bound11}
 |E(G)| \leq |E(G'')| + |C| + (n-|C|) \leq \frac{2cr+ 4r+7c+8}{2cr+6c}(n-|C|)+n.
\end{equation}
In the case $l=-1$, we obtain
\begin{equation}
\label{eqn:Approaching bound12}
 |E(G)| \leq |E(G'')| + |C| + (n-|C|) \leq \frac{6c+12}{5c+4}(n-|C|)+n.
\end{equation}

It remains to compare the above upper bounds on $|E(G)|$ and to optimize over the value of $|C|$. If $l>-1$, then the value of $|C|$ for which the right-hand sides of (\ref{eqn:Approaching bound5}) and (\ref{eqn:Approaching
bound11}) coincide is
\begin{displaymath}
   |C|= \frac{cr +3c}{cr+2r+ 4c+2}n.
\end{displaymath}
The claimed bound follows by plugging this value into
(\ref{eqn:Approaching bound5}) or (\ref{eqn:Approaching bound11}).

In the case $l=-1$, the critical value of $|C|$, obtained by
comparing the bounds (\ref{eqn:Approaching bound7})
and (\ref{eqn:Approaching bound12}), is
\begin{displaymath}
 |C|= \frac{5c + 4}{7c+8}n.
\end{displaymath}
Plugging this value into
(\ref{eqn:Approaching bound7}) or (\ref{eqn:Approaching bound12}),
the claimed bound follows.}\\

(C) {As before, let $C$ be a shortest odd cycle in $G$, and let $G'$ be the graph obtained from $G$ after doubling $C$. The doubled cycle is denoted by $C'\subset G'$. Let $G_0\supseteq C$ denote a {\em maximal} subgraph of $G$, which is turned into a {\em two}-connected subgraph of $G'$ after performing Conway's doubling on $C$. Let $G_1$ stand for the graph obtained from $G$ by the removal of all {\em edges} in $G_0$.

It is easy to see that $G_1$ is bipartite, and each of its connected components shares exactly one vertex with $G_0$. Indeed, if a connected component $G_2\subseteq G_1$ were not bipartite, then, by Lemma \ref{lemma:removing odd cycle}, $G_2$ would share at least one vertex with $C$, which belongs to an odd cycle of $G_2$. By the maximal choice of $G_0$, after doubling $C$, the component $G_2$ must turn into a subgraph $G_2'\subset G'$, which shares precisely {\em one} vertex with the doubled cycle $C'$. Thus, $G_2$ must also share precisely {\em one} vertex with $C$, which implies that $G_2'\subseteq G'$ has an odd cycle. This contradicts Lemma \ref{lemma:conway doubling}(i), according to which $G'$ is a bipartite graph.

Therefore, $G_1$ is the union of all blocks of $G$, which are not entirely  contained in $G_0$. Since each connected component $G_2$ of $G_1$ is bipartite, the number of edges of $G_2$ can be bounded from above by (\ref{eqn:Approaching bound1}), just like in Case (A).

In order to bound the number of edges of $G$, we proceed by adding
the connected components of $G_1$ to $G_0$, one by one. As at the end of the discussion of Case~(\ref{case:A}), using the fact that
 the last terms in (\ref{eqn:Approaching bound1}), which do not depend on $n$, are smaller than $-2$, we can complete
 the proof by induction on the number of connected components of $G_1$.}

\label{sec:The bound approaching}

\section{A better upper bound}\label{proof2nd}

As was pointed out in the Introduction, if we manage to prove that for any $l',\;-3 \leq l'\leq -1$, the dumbbell $\DB(6,6,l')$ is not thrackleable, then Theorem \ref{thm:Approaching bound} yields that the maximum number of edges that a thrackle on $n$ vertices can have is at most $\frac{617}{425}n<1.452n$. This estimate is already better than the currently best known upper bound $\frac{3}{2}n$ due to Cairns and Nikolayevsky \cite{Cairns}.

In order to secure this improvement, we have to exclude the subgraphs
$\DB(6,6,-1)$, $\DB(6,6,-2)$, and $\DB(6,6,-3)$. The fact that $\DB(6,6,-3)$ cannot be drawn as a thrackle was proved in \cite{Lovasz} (Theorem 5.1).
Here we present an algorithm that can be used for checking whether a ``reasonably'' small graph $G$ can be drawn as a thrackle. We applied our algorithm to verify that $\DB(6,6,-1)$ and $\DB(6,6,-2)$ are indeed not thrackleable. In addition, we show that $\DB(6,6,0)$ cannot be drawn as thrackle, which leads to the improved bound in Theorem \ref{thm:betterBound}.

\medskip

\begin{figure}[h]
\centering
\includegraphics[scale=0.4]{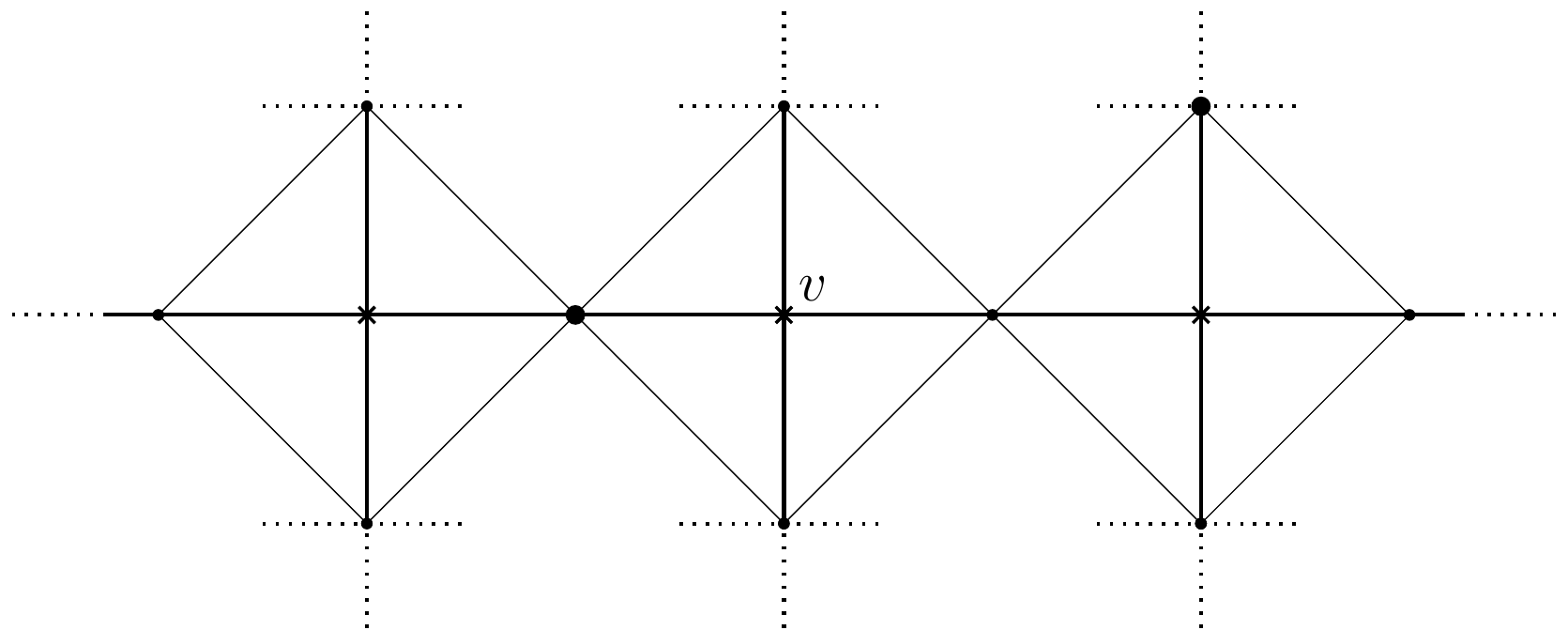}
\caption{4-cycle around a vertex $v$ of $G'$, which was a crossing point in $G$}
\label{fig:crossing}
\end{figure}

Let $G=(V,E)$ be a thrackle. Direct the edges of $G$ arbitrarily.
 For any $e\in E$, let $E_e\subseteq E$ denote the set of all edges
 of $G$ that do not share a vertex with $e$, and let $m(e)=|E_e|$.
 Let $\pi_e=(\pi_e(1), \pi_e(2), \ldots, \pi_e(m(e))$ stand for the $m(e)$-tuple
 (permutation) of all edges belonging to $E_e$, listed in the order of their crossings
along $e$.

Construct a planar graph $G'$ from $G$, by introducing a new vertex at each 
crossing between a pair of edges of $G$, and replacing each edge by its pieces.
 In order to avoid $G'$ having an embedding in which two paths corresponding to
 a crossing pair of edges of $G$ do not {\em properly} cross, we introduce a new
 vertex in the interior of every edge of $G'$, whose both endpoints are former crossings.
 For each former crossing point $v$, we add a cycle of length {\em four} to $G'$, connecting
 its neighbors in their cyclic order around $v$, as illustrated in Figure \ref{fig:crossing}.
 In the figure, the thicker lines and points represent edges and vertices or crossings of $G$,
 while the thinner lines and points depict the {\em four}-cycles added at the second stage.

Obviously, $G'$ is completely determined by the directed abstract underlying graph of $G$ and by the set of permutations $\Pi(G):= \{ \pi_e \in E_e^{m(e)}| e\in E\}$. Thus, a graph $G=(V,E)$ can be drawn as a thrackle if and only if there exists a set $\Pi$ of $|E|$ permutations of
$E_e, e\in E,$ such that the abstract graph $G'$ corresponding to the pair $(G,\Pi)$ is {\em planar}. In other words,
to decide whether a given abstract graph $G=(V,E)$ can be drawn as a thrackle, it is enough to consider all possible
sets of permutations $\Pi$ of $E_e, e\in E$, and to check if the corresponding graph $G'=G'(G,\Pi)$ is planar for at least one of them.
The first deterministic linear time algorithm for testing planarity was found by Hopcroft and Tarjan \cite{hopcroft}.
However, in our implementation we
used an improved algorithm for planarity testing by Fraysseix et al. \cite{Fraysseix}, in particular,
its implementation in the library P.I.G.A.L.E. \cite{Pigale}.
We leave the pseudocode of our routine for the abstract. The source code can be found here : 
http://dcg.epfl.ch/webdav/site/dcg/users/183292/public/Thrackle.zip.

It was shown in $\cite{Lovasz}$ (Lemma 5.2) that in every drawing of a directed cycle $C_6$ as a thrackle, either every oriented path $e_1e_2e_3e_4$ is drawn in such a way that $\pi_{e_1}=(e_4,e_3)$ and $\pi_{e_4}=(e_1,e_2)$, or every oriented path $e_1e_2e_3e_4$ is drawn in such a way that $\pi_{e_1}=(e_3,e_4)$ and $\pi_{e_4}=(e_2,e_1)$. Using this observation (which is not crucial, but saves computational time), we ran a backtracking algorithm to rule out the existence of a set of permutations $\Pi$,
for which $G'(\DB(6,6,0),\Pi)$, $G'(\DB(6,6,-1),\Pi)$, or $G'(\DB(6,6,-2),\Pi)$ is planar.
 Our algorithm attempts to construct larger and larger parts of a potentially good set $\Pi$,
 and at each step it verifies if the corresponding graph still has a chance to be extended to
 a planar graph. In the case of $\DB(6,6,0)$, to speed up the computation, we exploit Lemma 2.2 from \cite{Lovasz}.

Summarizing, we have the following

\begin{lemma}
\label{lemma:forbiddenGraphs}
None of the dumbbells $\DB(6,6,l')$,  $-3\le l'\le 0$ can be drawn as a thrackle.
\end{lemma}

According to Lemma~\ref{lemma:forbiddenGraphs}, Theorem~\ref{thm:Approaching bound} can be
 applied with $c=6, l=0$, and Theorem~\ref{thm:betterBound} follows.

For any $\eps>0$, our Theorem~\ref{thm:Approaching bound} and the above observations 
provide a deterministic algorithm with bounded running time to prove that all thrackles
 with $n$ vertices have at most $(1+\eps)n$ edges or to exhibit a counterexample to Conway's conjecture.

In what follows, we estimate the dependence of the running time of our algorithm on $\eps$. 
The analysis uses the standard random access machine model.
 In particular, we assume that all basic arithmetic operations can be carried out in constant time.

\begin{theorem}
\label{thm:runningTime}
 For any $\eps>0$, there is a deterministic algorithm with running time $e^{O((1/\eps^2)\ln(1/\eps))}$ to prove that all thrackles with $n$ vertices have at most $(1+\eps)n$ edges or to exhibit a counterexample to Conway's conjecture.
\end{theorem}

\begin{proof}
First we estimate how long it takes for a given $c$ and $l$, satisfying the assumptions in Theorem \ref{thm:Approaching bound},
to check whether there exists a dumbbell $\DB(c',c'',l')$ with $c'$ and $c''$ even, $6\leq c' \leq c'' \leq c$,
and  with $-c'/2\leq l'\leq l$,
that can be drawn as a thrackle. Clearly, there are
$$\sum_{\begin{subarray}{c}
c'=6\\
c' \ \mathrm{is} \ \mathrm{even}
\end{subarray}}^{c}\frac{(\frac{c'}{2}+l+1)(c-c'+2)}{2}= \frac{1}{8}lc^2+\frac{1}{48}c^3-\frac{3}{4}lc+l+\frac{1}{4}c^2-\frac{25}{12}c+3\leq
\kappa(lc^2+c^3)$$
dumbbells to check, for some $\kappa>0$.
In order to decide, whether a fixed dumbbell with $m$ edges can be drawn as a thrackle, we construct at most $(m-2)!^{m}$ graphs, each with at most $O(m^2)$ edges, and we test each of them for planarity. Thus, the total running time of our algorithm is $O((lc^2+c^3)(2c+l-2)!^{2c+l}(2c+l)^2)$.
Approximating the factorials by Stirling's formula, we can conclude that the running time is $O((2c+l)^{(2c+l)^2+\frac{1}{2}(2c+l)+5}e^{-(2c+l)})$.

Now, for any $1>\epsilon>0$ we show how big values of $l$ and $c$ we have to take so that 
Theorem \ref{thm:Approaching bound} gives the upper bound $(1+\epsilon)n$ on the maximum number of edges in a thrackle.
We remind the reader that $r=\lfloor l/2\rfloor$. 
It can be shown by routine calculation that there are three constants  $\kappa, \kappa_r$ and $\kappa_c$ so that the following holds.
Given $\epsilon>0$,  for $r= \lceil\frac{\kappa_r}{\eps}\rceil$,
and $c$ such that $$c\geq \frac{\kappa_c}{\eps} \geq \frac{\kappa r^2}{\epsilon(2r^2+14r+24)-2r-7}$$
the value of $\tau(c,l)$ introduced in Theorem \ref{thm:Approaching bound} is at most $1+\eps$.
For the sake of completeness we give the sufficient condition for $c$ only in terms of $r$ and $\epsilon$:

$$c\geq \frac{r^2(2-2\epsilon)+r(11-8\epsilon)+11-6 \epsilon+(r+3)\sqrt{(r^2(4
 +8 \epsilon+4 \epsilon^2)+ r(4+36 \epsilon+ 8 \epsilon^2)+1+28 \epsilon+4 \epsilon^2)}}{\epsilon(2r^2+14r+24)-2r-7}
$$
 Thus, for these values of $c$ and $r$ 
Theorem  \ref{thm:Approaching bound} gives the required bound,
i.e. at most $(1+\epsilon)n$. Plugging  $\frac{\kappa_c}{\eps}$ and
 $2\frac{\kappa_r}{\eps}$ as $c$ and $l$, respectively, in  $O((2c+l)^{(2c+l)^2+\frac{1}{2}(2c+l)+5}e^{-(2c+l)}),$
 the theorem follows.
\end{proof}

\section{Concluding remarks}

We say that two cycles $C_1$ and $C_2$ of a graph are at distance $l\ge 0$, if the length of a shortest path joining a vertex of $C_1$ to a vertex of $C_2$ is $l$. The following Tur\'an-type questions were motivated by the proof of Theorem~\ref{thm:Approaching bound}.

\noindent(1) Given two integers $c_1, c_2$, with $3\leq c_1\leq c_2$, what is the maximum number of edges that a planar graph on $n$ vertices can have, if its girth is at least $c_1$, and no two cycles of length at most $c_2$ share an edge?

\noindent(2) Given three integers $c_1, c_2$, and $l$, with $3\leq c_1\leq c_2$ and $l\ge 0$, what is the maximum number of edges that a planar graph on $n$ vertices can have, if its girth is at least $c_1$, and any two of its cycles of length at most $c_2$ are at distance larger than $l$ ?

 The inequalities (\ref{eqn:Approaching bound1}) provide nontrivial upper bounds for restricted versions of the above problem for bipartite graphs.

We have the following general result.

\begin{theorem}
\label{theorem:TuranUpperBound}
Let $c_1, c_2$,  and $l$ denote three non-negative natural numbers with $3\leq c_1\leq c_2$. Let $G$ be a planar graph with $n$ vertices and girth at least $c_1$.
\begin{enumerate}[(i)]
\item
If no two cycles of length at most $c_2$ share an edge, then  $|E(G)|\le \frac{c_1c_2+c_1}{c_1c_2-c_2-1}n$.
\item
If no two cycles of length at most $c_2$ are at distance at most $l$, then \\$|E(G)|\le \frac{c_1c_2+2\lfloor l/2\rfloor c_2 +2\lfloor l/2 \rfloor +c_2+1}{2\lfloor l/2 \rfloor c_2 - 2\lfloor l/2\rfloor+c_1c_2-c_1}n.$
\end{enumerate}
\end{theorem}

\begin{proof} (Outline.)
Without loss of generality, we can assume in both cases that $G$ is connected, it has no vertex of degree {\em one}, and it is not a cycle.
To establish part (i), consider an embedding of $G$ in the plane. Let $m=|E(G)|$, and let $f$ and $f_{c_2}$ stand for the number of faces of $G$ and for the number of faces of length at most $c_2$. We follow the idea of the proof of Case (A), Theorem \ref{thm:Approaching bound}, with $f_{c_2}\leq \frac{1}{c_1}m$ instead of $f_c\leq \frac{1}{6}m$, and with the
inequality
\begin{equation}
\nonumber
2m\geq c_1f_{c_2}+(c_2+1)(f-f_{c_2})
\end{equation}
replacing (\ref{eqn:Approaching bound100}).
Analogously, in the proof of part (ii), we use $f_{c_2}\leq \frac{1}{2\lfloor l/2 \rfloor+c_1}n$ instead of the inequality $f_c\leq \frac{1}{2r+6}n$.
\end{proof}

It is possible that the constant factor in the part (i) of Theorem~\ref{theorem:TuranUpperBound}
 is tight for all values of $c_1$ and $c_2$. It is certainly tight for all values of the form 
$c_1=ml$ and $c_2=m(l+1)-1$, where $m$ and $l$ are natural numbers, as is shown by the following result.

\begin{theorem}
\label{theorem:construction}
For any positive integers $n_0$, $m\geq 1$, and $l\geq 3$, one can construct a plane graph $G=(V,E)$ on at least $n_0$ vertices with girth $ml$
such that all of its inner faces are of size $ml$ or $m(l+1)$, its outer face is of size $2ml$, and each edge of $G$ not on its outer face belongs to exactly one cycle of size $ml$, which is a face of $G$. The second smallest length of a cycle in $G$ is $m(l+1)$.
\end{theorem}

\begin{proof}
\begin{figure}[h]
\centering
\includegraphics[scale=0.5]{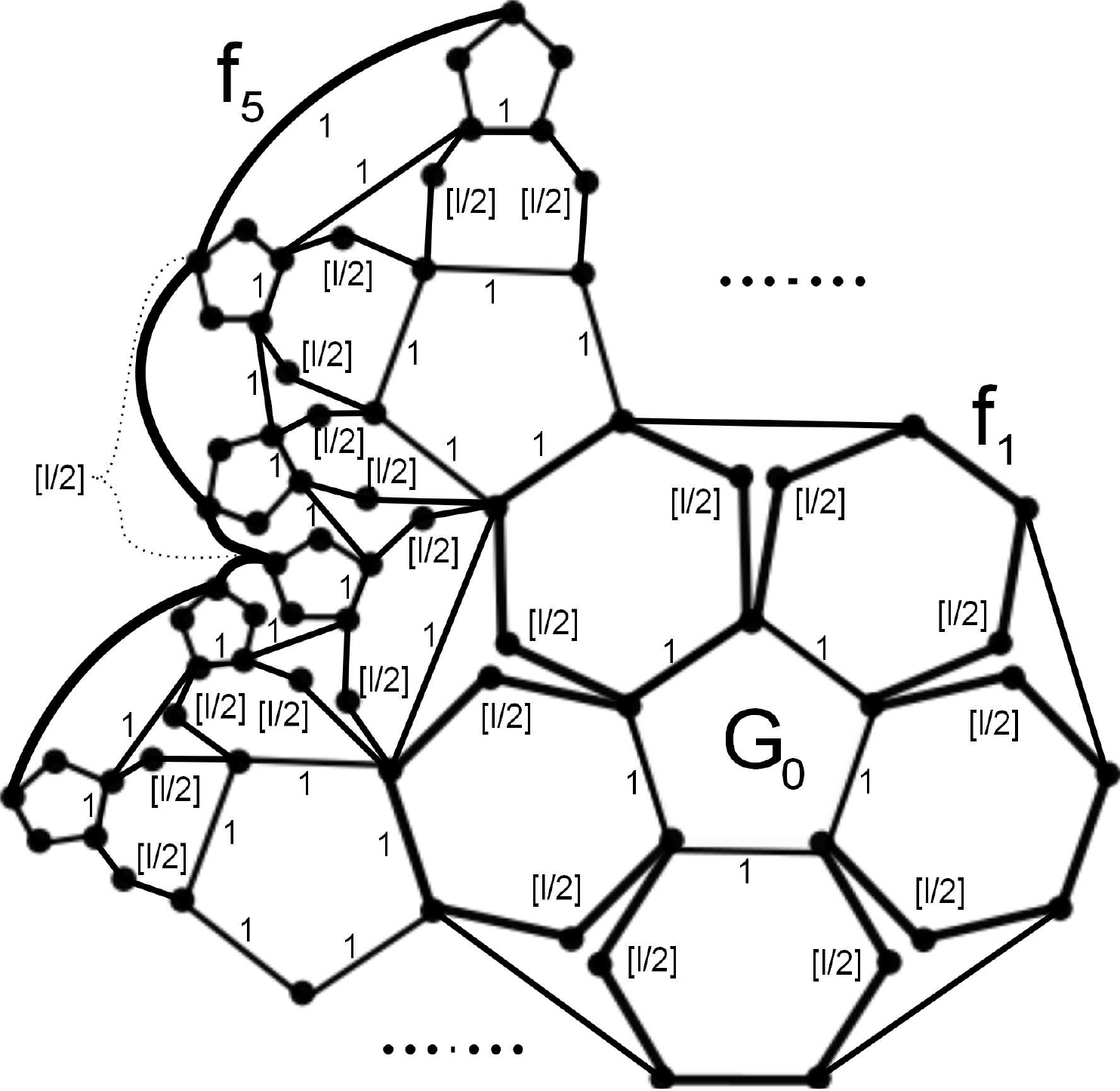}
\caption{The key part of the construction from the proof of Theorem \ref{theorem:construction} for $l=5$, and $m=1$}.
\label{fig:construction4}
\end{figure}

\begin{figure}[h]
\centering
\includegraphics[scale=0.5]{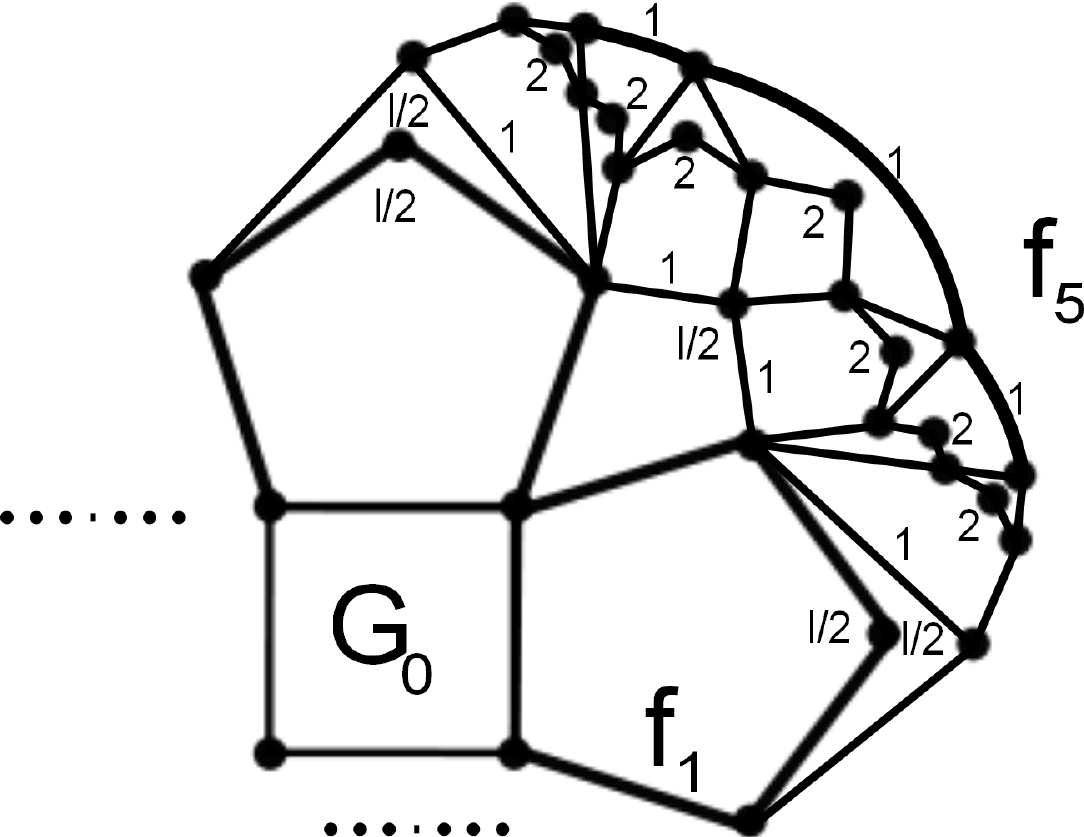}
\caption{The key part of the construction from the proof of Theorem \ref{theorem:construction} for $l=4$, and $m=1$}.
\label{fig:construction44}
\end{figure}

Intuitively, one can think of a graph $G$ meeting the requirements of the theorem as a ``generalized chessboard'' with white and black fields (faces) of size $ml$ and $m(l+1)$, respectively.

Observe that it is enough to provide a construction for $m=1$. Indeed, given a construction $G$ for some $l=l'$, $n_0=n_0'$, and $m=1$, for any $m'>1$, one can subdivide each edge of $G$ into $m'$ pieces to obtain a valid construction for $l:=l'$, $n_0:=n_0'$, and $m:=m'$.

Here we consider only the case when $l$ is {\em odd}; the other case can be treated analogously. We construct $G$ recursively, starting from a plane graph $G_0$, 
which is a cycle of length $l$, as depicted in Figure \ref{fig:construction4}. Let $f_i$ denote the outerface of $G_i$, $i=0,1,2,\ldots$ (for $i>1$, the outer faces $f_i$ are not completely depicted in the figure). Our construction satisfies the condition that each edge of $G$ lies exactly on one outerface $f_i$ for some $i$.

For any $i\ge 0$, we obtain $G_{2i+1}$ from $G_{2i}$, by attaching
faces of size $l+1$ along $f_{2i}$ in the way indicated in Figure \ref{fig:construction4} (the labels of the paths in the figure indicate the length).
Analogously, for any $i\ge 1$, the graph $G_{2i}$ can be obtained from $G_{2i-1}$, by attaching to the sides of $f_{2i-1}$
faces of size $l$, in the way indicated in the figure. Observe that Figure \ref{fig:construction4} can be easily
modified to work for any odd value of $l$, and it is not hard to obtain similar construction for even values either (see Figure \ref{fig:construction44}). The key feature of the construction is that the length of the outer face $f_1$ is the same as the length of the outer face $f_5$ (both are drawn with thicker lines in the figure). Thus, we can repeat the pattern consisting of the outer faces $f_1,\ldots, f_5$ until the number of vertices in $G$ is at least $n_0$, and then finish with a graph $G_{4i+2}$, for some $i$.
Notice, that during this process we never create multiple edges, and that each edge lies on the outer face of exactly one $G_i$.

In what follows, we show that the girth of $G$ is $l$ and that no two cycles of length $l$ share an edge, which concludes the proof. To this end we show that a smallest cycle $C$ in $G$
is a face cycle of length $l$.
 In order to see this we proceed by distinguishing two cases. 

First consider the case, when $C$ contains vertices belonging to the outer face $f_{4i+1}$, and
vertices belonging to the outer face $f_{4i+5}$, for some $i$. In this case we are done,
since a shortest path between $f_{4i+1}$ and  $f_{4i+5}$ is of length $l-1$. Otherwise we can proceed by checking a small subgraph of $G$.
\end{proof}

If we slightly relax the conditions in Theorem \ref{theorem:TuranUpperBound} by forbidding only dumbbells determined by {\em face cycles},
 we obtain some tight bounds. For instance, it is not hard to prove the following.

\begin{theorem}
\label{theorem:TuranUpperBound2}
Let $c_1$ and $c_2$ be two nonnegative integers with $3\leq c_1\leq c_2$. Let $G$ be a plane graph on $n$ vertices that 
has no face shorter than $c_1$ and no two faces of length at most $c_2$ that share an edge. Then we have
$|E(G)|\le \frac{c_1c_2+c_1}{c_1c_2-c_2-1}n,$
and the inequality does not remain true with any smaller constant.
\end{theorem}


\newpage
\appendix{{\bf Appendix}}

\section{Backtracking algorithm}

For sake of completeness 
in this section we describe a backtracking algorithm checking, whether a given dumbbell $G=(V,E)$ can be drawn as a thrackle.
We orient the edges of $G$, so that we can traverse them by a single walk, so called Euler's walk, during which we visit each edge just once. 
We use the notation from Section \ref{proof2nd}.

Let us start with a description of the routines used by our algorithm.

The routine  UPDATE($\pi_e$, $e'$, pos)
 returns the updated permutation
 $\pi_e 
\in 
{E'}_{e}^{
m'(e)}$,
 which corresponds to adding one more crossing vertex to an already constructed
 part of $(G',\Pi')$ corresponding to a subgraph $G'=(V',E')$ of $G$,
 where $e\in E$, $e'\in E'$, $m'(e)$ returns the number of crossings
 of $e$ already modeled by $(G',\Pi')$, and  $
\Pi'(G')
:=
 \{ \pi_e 
\in
 {E'}_e^{
m'(e)}| e\in E'\}$.
UPDATE($\pi(e)$, $e'$, pos) returns the permutation $\pi_e'$ whose length is by one longer than $\pi_e$, such that

\begin{displaymath}
\pi'_e(i) = \left\{ \begin{array}{ll}
                \pi_e(i) &   {\mbox {if }}   i<pos \\ 
                e' &  {\mbox {if }}    i=pos \\
				\pi_{e}(i-1)&   {\mbox {if }} 	  i >pos 	
                \end{array}
                 \right.
\end{displaymath}
REVERSE\_UPDATE($\pi_e$, $e'$, $pos$) corresponds to the reverse operation of the operation UPDATE($\pi_e$, $e'$, pos).
PICK\_NEXT\_EDGE(G) returns a next edge in our Euler's walk.
 In order to check, whether $G$ can be drawn as a thrackle the algorithm just calls the procedure BACKTRACKING($e$) for an edge $e\in E$.
 The algorithm returns {\it true} if $G$ can be drawn as a thrackle, and it returns {\it false} if $G$ cannot be drawn as a thrackle.
 In our description of the algorithm we restrain from all optimization details, which were mentioned in Section \ref{proof2nd}.
The pseudocode of the backtracking routine follows.

{\begin{algorithm}[H]
  \label{a1}
  \caption{Thrackleabilty testing} 
  {BACKTRACKING ($e\in E(G)$) \;} 
  \Begin{
	 \uIf{$(G',\Pi')$ cannot be extended }{
		\Return true 
	 }

	  \uIf{e = -1}{
		$e$  = PICK\_NEXT\_EDGE(G) 
	 }
	 
	 \uIf{$e$ has crossed all  edges in $E_e'$}{
			BACKTRACKING(-1) 
	 }
	 
	 \uElse{

    \ForAll{$e'\in E_e'$ which $e$ has not already crossed} {
		\For{pos = 1 to length($\pi_{e'}$)} {
			$\pi_{e'}$ = UPDATE($\pi_{e'}$, $e$, $pos$) \;
			$\pi_{e}$ = UPDATE($\pi_{e}$, $e'$, LENGTH($\pi_e$)+1) \;
			\uIf {IS\_PLANAR(($G',\Pi'$))}
			{
				\uIf {BACKTRACKING($e$)}
				{
					\Return true
				}
				\uElse
				{
					REVERSE\_UPDATE($\pi_{e'}$, $e$, $pos$) \;
					
					REVERSE\_UPDATE($\pi_e$, $e'$, LENGTH($\pi_e$))  
				}
			}	
				\uElse
				{
					REVERSE\_UPDATE($\pi_{e'}$, $e$, $pos$) \;
					
					REVERSE\_UPDATE($\pi_e$, $e'$, LENGTH($\pi_e$))  
				}

		}
    }
    }
    \Return false   }
\end{algorithm}}

\end{document}